\documentclass[a4paper,10pt]{amsart}
\usepackage{amssymb,amscd,amsmath}

\theoremstyle{plain}
\newtheorem{thm}{Theorem}[section]
\newtheorem{prop}[thm]{Proposition}
\newtheorem{lem}[thm]{Lemma}

\theoremstyle{definition}

\newtheorem{prob}{Problem}[section]

\theoremstyle{remark}
\newtheorem{rem}{Remark}[section]
\newtheorem{ex}[rem]{Example}

\newcommand{\forme}[1]{}

\newcommand{\Ot}{{\mathbf O}_{\mathrm \theta}}
\newcommand{\Or}{{\mathbf O}^{\mathrm \theta}}

\title{A family of non-Schurian $p$-Schur rings over groups of order $p^3$}

\author[K.~Kim]{Kijung Kim}
\address{Department of Mathematics, Pusan National University, Busan 609-735, Republic of Korea}
\email{knukkj@pusan.ac.kr}

\date{\today}

\begin{document}
\maketitle

\begin{abstract}
Recently, it was proved that every commutative $p$-Schur ring over a group of order $p^3$ is Schurian.
In this article, we consider the Schurity problem of non-commutative $p$-Schur rings over groups of order $p^3$.
In particular, it is given a family of non-Schurian $p$-Schur rings over groups of order $p^3$.
\end{abstract}

\footnotetext{This research was supported by Basic Science Research Program through the National Research Foundation of Korea(NRF) funded by the Ministry of Education (2013R1A1A2005349).}

\section{Introduction}\label{sec:intro}

Let $H$ be a finite group.
We denote by $\mathbb{Q}H$ the group algebra of $H$ over the rational number field $\mathbb{Q}$.
For a nonempty subset $T \subseteq H$, we set $\underline{T}:= \sum_{t \in T}t$ which is treated as an element of $\mathbb{Q}H$.

A subalgebra $\mathcal{A}$ of the group algebra $\mathbb{Q}H$ is called a \textit{Schur ring} (briefly \textit{S-ring}) over $H$
if there exists a partition $\mathrm{Bsets}(\mathcal{A}):= \{ T_0, T_1, \dotsc, T_r \}$ of $H$ satisfying the following conditions:
\begin{enumerate}
\item $\{ \underline{T_i} \mid T_i \in \mathrm{Bsets}(\mathcal{A}) \}$ is a linear basis of $\mathcal{A}$;
\item $T_0 = \{ 1_H\}$;
\item $T_i^{-1}:=\{ t^{-1} \mid t \in T_i \} \in \mathrm{Bsets}(\mathcal{A})$ for all $T_i \in \mathrm{Bsets}(\mathcal{A})$.
\end{enumerate}
A S-ring $\mathcal{A}$ over a $p$-group $H$ is called a \textit{$p$-Schur ring} (briefly \textit{$p$-S-ring})
if the size of every element in $\mathrm{Bsets}(\mathcal{A})$ is a power of $p$, where $p$ is a prime.

Let $G$ be a subgroup of $\mathrm{Sym}(H)$ containing the right regular representation of $H$.
We denote by $T_0 = \{1_H \}, T_1 , \dotsc, T_r$ the orbits of the stabilizer $G_{1_H}$.
The \textit{transitivity module} $V(H, G_{1_H})$ of $G$ is the vector space spanned by $\{\underline{T_i} \mid 0 \leq i \leq r\}$.
It was proved in \cite{wie} that $V(H, G_{1_H})$ is a S-ring over $H$.
Customarily, an S-ring $\mathcal{A}$ over $H$ is called \textit{Schurian} if $\mathcal{A}$ is the transitivity module $V(H, G_{1_H})$
for some group $G$ containing the right regular representation of $H$.
A family of S-rings which are not Schurian was given in \cite[Theorem 26.4]{wie}.
It is known that every S-ring over a cyclic $p$-group is Schurian (see \cite{mpschur}).
In $1979$, M. Klin conjectured that every S-ring over a cyclic group is Schurian.
But, it was proved in \cite{pono} that there exist non-Schurian S-rings over cyclic groups.
Recently, it was proved that every commutative $p$-S-ring over a group of order $p^3$ is Schurian (see \cite{kim2013,kim2014,sw}).
In this article, we consider the Schurity problem of non-commutative $p$-Schur rings over groups of order $p^3$.
First of all, we observe non-commutative $7$-S-rings over a group of order $7^3$.
\begin{ex}\label{s-rings:example}
Let $H_0 = \langle a, b \mid a^{7^2} = b^7 = 1, ab = ba^{8} \rangle$ be a non-abelian group of order $7^3$.
Then the following partitions give two $7$-S-rings $\mathcal{A}_i$ over $H_0$ $(i= 1, 2)$:
\[\mathrm{Bsets}(\mathcal{A}_i)= \{ \{l \} \mid l \in \langle a^7, b \rangle \}
\cup \bigcup_{1 \leq j \leq 6} \{ a^j \langle b(a^7)^{x_i(j)} \rangle (a^7)^k \mid 0 \leq k \leq 6 \},\]
where $(x_1(1), x_1(2), x_1(3), x_1(4), x_1(5), x_1(6)) = (0,3,6,2,5,1)$

and $(x_2(1), x_2(2), x_2(3), x_2(4), x_2(5), x_2(6)) = (0,4,2,5,6,1)$.
\end{ex}
In Example \ref{s-rings:example}, the action of $\langle \phi \rangle$ on $H_0$ induces $\mathrm{Bsets}(\mathcal{A}_1)$,
where $\phi : a \mapsto ab, b \mapsto b$ is an automorphism of $H_0$.
This means that $\mathcal{A}_1$ is Schurian.
But, no automorphism of $H_0$ induces $\mathrm{Bsets}(\mathcal{A}_2)$.
So, one might be curious what can be said in the case of $\mathcal{A}_2$.
An attempt to verify the Schurity of $\mathcal{A}_2$ leads to our results.

In Example \ref{s-rings:example}, the group generated by $\{ T^{-1}T \mid T \in \mathrm{Bsets}(\mathcal{A}_i)\}$ $(i=1,2)$
is an elementary abelian group of rank $2$.
We denote it by $\Or(\mathcal{A}_i)$.
Two $7$-S-rings have common properties as follows:
\begin{enumerate}
\item For each $T \in \mathrm{Bsets}(\mathcal{A}_i)$, $T^{-1}T$ is a proper subgroup of $\Or(\mathcal{A}_i)$;
\item For $T, T' \in \mathrm{Bsets}(\mathcal{A}_i)$ with $T\Or(\mathcal{A}_i) \neq T'\Or(\mathcal{A}_i)$,
$T^{-1}T$ and $T'^{-1}T'$ are different.
\end{enumerate}

In general, we deal with $p$-S-rings satisfying the above properties.
Our first step is to determine $\mathrm{Bsets}(\mathcal{A})$.
Let $\mathcal{A}$ be a $p$-S-ring over a non-abelian group $H$ of order $p^3$ satisfying the above properties (see $(\ref{A})$ and $(\ref{B})$ for details).
If $T_1, \dotsc, T_{p-1}$ are elements of $\mathrm{Bsets}(\mathcal{A})$ such that $T_i^{-1}T_i \neq T_j^{-1}T_j$ for distinct $1 \leq i,j \leq p-1$,
then we have
\[\mathrm{Bsets}(\mathcal{A}) = \{ \{l \} \mid l \in \Or(\mathcal{A}) \} \cup \bigcup_{1 \leq i \leq p-1} \{ T_iz^j \mid 0 \leq j \leq p-1 \},\]
where $z$ is a nontrivial element of the center of $H$.
Based on this result, we are able to give a criterion for Schurity of $\mathcal{A}$ and
verify the existence of non-Schurian $p$-S-rings over groups of order $p^3$.

This article is organized as follows.
In Section \ref{sec:pre}, we review definitions and known facts about S-rings.
In Section \ref{sec:main}, we give a criterion when $p$-S-rings over groups of order $p^3$ with some conditions are Schurian.
In Section \ref{sec:comb}, we introduce a problem related to constructions of $p$-S-rings.
In Section \ref{sec:nonschur}, we construct non-Schurian $p$-S-rings over groups of order $p^3$.

\section{Preliminaries}\label{sec:pre}
Throughout this article, we use the notations based on \cite{mpschur,zies2}.

For a group $H$, we denote by $R_H$ the set of all binary relations on $H$
invariant with respect to the right regular representation of $H$.
Then the mapping
\[ 2^H \rightarrow R_H  ~(T \mapsto R_H(T)), \]
where $R_H(T) = \{ (h, th) \mid h \in H, t \in T \}$, is a bijection.
If $\mathcal{A}$ is a S-ring over $H$, then the pair
\[ \mathcal{C}(\mathcal{A}) = (H, R_H(\mathrm{Bsets}(\mathcal{A}))), \]
where $R_H(\mathrm{Bsets}(\mathcal{A})) = \{ R_H(T) \mid T \in \mathrm{Bsets}(\mathcal{A}) \}$,
is called a \textit{Cayley (association) scheme} over $H$ (see Subsection \ref{subsec:scheme} for association schemes).

Let $\mathcal{C} = (H, R)$ be a Cayley scheme.
For each $r \in R$, we set
\[r(1_H) = \{h \in H \mid (1_H, h) \in r \}.\]
Then the vector space spanned by $\{ \underline{r(1_H)} \mid r \in R \}$ is a S-ring over $H$.

\begin{thm}[See \cite{mkk}]\label{mk}
The correspondence $\mathcal{A} \mapsto \mathcal{C}(\mathcal{A}), \mathcal{C}(\mathcal{A}) \mapsto \mathcal{A}$
induces a bijection between the S-rings and Cayley schemes over the group $H$ that preserves the natural partial orders on these sets.
\end{thm}

Theorem \ref{mk} means that the S-rings and Cayley schemes provide two different ways to talk about the same thing.
In the rest of this section, we consider definitions in the theory of association schemes and translate them into S-rings.

\subsection{Association schemes}\label{subsec:scheme}

Let $X$ be a nonempty finite set.
Let $S$ denote a partition of $X \times X$. Then we say that $(X, S)$ is
an \textit{association scheme} (briefly  \textit{scheme}) if it satisfies the following conditions:
\begin{enumerate}
\item $1_X := \{ (x, x) \mid  x \in X \} \in S$;
\item For each $s \in S$, $s^* := \{(x, y) \mid (y, x) \in s \} \in S$;
\item For all $ s, t, u \in S$ and $x, y \in X$, $a_{stu} :=| xs \cap yt^* |$ is constant whenever $(x, y) \in u$,
where $xs:=\{ y \in X \mid (x,y) \in s \}$.
\end{enumerate}
For each $s \in S $, we abbreviate $ a_{ss^*1_X}$ as $n_s$, which is called the \textit{valency} of $s$.
Put $n_S = \sum_{s \in S} n_s$.
We call $n_S$ the \textit{order} of $S$.
A scheme $(X,S)$ is called \textit{$p$-valenced} if the valency of every element is a power of $p$, where $p$ is a prime.
In particular, a $p$-valenced scheme $(X,S)$ is called a \textit{$p$-scheme} if $n_S$ is also a power of $p$.

Let $P$ and $Q$ be nonempty subsets of $S$. We define $PQ$ to be the set of all elements $s \in S$ such that there exist elements
$p \in P$ and $q \in Q$ with $a_{pqs}\neq 0$. The set $PQ$ is called the \textit{complex product} of $P$ and $Q$.
If one of the factors in a complex product consists of a single element $s$, then one usually writes  $s$ for $\{ s \}$.

A nonempty subset $T$ of $S$ is called \textit{closed} if $TT \subseteq T$.
Note that a subset $T$ of $S$ is closed if and only if $\bigcup_{t \in T}t$ is an equivalence relation on $X$.
A closed subset $T$ is called \textit{thin} if all elements of $T$ have valency 1.
The set $\{ s \in S \mid n_s=1 \}$ is called the \textit{thin radical} of $S$ and denoted by $\Ot(S)$.

For binary relations $r, s \subseteq X \times X$, we set $r \cdot s = \{ (\alpha,\gamma) \mid (\alpha,\beta) \in r, (\beta,\gamma) \in s ~~\text{for some}~~ \beta \in X\}$. We call $r \cdot s$ the relational product of $r$ and $s$.
Note that $T$ is thin if and only if $T$ is a group under the relational product.

A closed subset $T$ of $S$ is called \textit{strongly normal} in $S$, denoted by $T \lhd^\sharp S$, if $s^* T s \subseteq T$ for every $s \in S$.
We put $\Or(S) := \bigcap_{T \lhd^\sharp S} T $ and call it the \textit{thin residue} of $S$.
Note that $\Or(S) = \langle \bigcup_{s \in S} s^*s \rangle$ (see \cite[Theorem 2.3.1]{zies}).

Let $(X,S)$ and  $(X_1, S_1)$ be association schemes. A bijective map $\phi$ from $X \cup S$ to $X_1 \cup S_1$ is called an \textit{isomorphism} if it satisfies the following conditions:
\begin{enumerate}
\item $X^\phi \subseteq X_1$ and $S^\phi \subseteq S_1$;
\item For all $x, y \in X$ and $s \in S$ with $(x,y) \in s$, $(x^\phi, y^\phi) \in s^\phi$.
\end{enumerate}
An isomorphism $\phi$ from $X \cup S$ to $X \cup S$ is called an \textit{automorphism} of $(X,S)$
if $s^{\phi}=s$ for all $s \in S$.
We denote by $\mathrm{Aut}(X,S)$ the automorphism group of $(X,S)$.

\begin{lem}[See \cite{zies}]\label{lem:constant}
Let $(X, S)$ be an association scheme.
For $u, v, w \in S$, we have the following:
\begin{enumerate}
\item $a_{uwv} n_v = a_{u^\ast vw} n_w = a_{v w^\ast u} n_u$;
\item $n_u n_v = \sum_{s \in S} a_{uvs} n_s$.
\end{enumerate}
\end{lem}

\begin{lem}[See \cite{mp1}]\label{lem:rel}
Let $(X,S)$ be an association scheme and $r, s \in S \setminus \{1_X\}$.
Then $rr^* \cap ss^* = \{ 1_X \}$ if and only if $a_{r^*st} \leq 1$ for all $t \in S$.
\end{lem}

For a binary relation $t \subseteq X \times X$ and subsets $E, F \subseteq X$, we define the adjacency matrix $A_t$ of $t \cap (E \times F)$
whose rows and columns are indexed by the elements of $E$ and $F$ as follows:
\begin{equation*}\label{product}
(A_t)_{xy} = \left\{
                      \begin{array}{ll}
                      1 & \hbox{if $(x,y) \in t$;} \\
                      0 & \hbox{otherwise.}
                      \end{array}
                     \right.
\end{equation*}
If both of $E$ and $F$ are $X$, then we usually write $A_t$ without mentioning $t \cap (E \times F)$.

\begin{lem}\label{lem:intersection}
Let $(X,S)$ be an association scheme.
Fix $x \in X$. For $d, e, f \in S$,
each column of the adjacency matrix of $e \cap (xd \times xf)$ has $a_{def}$ $1$'s.
\end{lem}
\begin{proof}
Let $A_e$ be the adjacency matrix of $e \cap (xd \times xf)$.
Fix $y \in xf$.
We count the number of $x'$ such that $(A_e)_{x'y}=1$.
Then it amounts to $|xd \cap ye^\ast|$.
By the definition of association schemes, we have $|xd \cap ye^\ast| = a_{def}$.
\end{proof}

\begin{lem}\label{lem:inter-relation}
Let $(X,S)$ be an association scheme.
Fix $x \in X$. For $d, f \in S$, we have $\{s \in S \mid s \cap (xd \times xf) \neq \emptyset \} = \{ s \in S \mid s \in d^\ast f\}$.
\end{lem}
\begin{proof}
By Lemma \ref{lem:intersection}, we have $s \cap (xd \times xf) \neq \emptyset \Leftrightarrow a_{dsf} \neq 0$.
By Lemma \ref{lem:constant} (i), we have $a_{dsf} n_f = a_{d^\ast fs} n_s$.
Then $a_{dsf} \neq 0 \Leftrightarrow a_{d^\ast fs} \neq 0$.
This completes the proof.
\end{proof}

\begin{thm}[See \cite{kim2011}]\label{kim11}
Let $(X,S)$ be a non-commutative $p$-scheme of order $p^3$ which is not thin.
Then $\Or(S)$ is a thin closed subset of $S$ isomorphic to $C_p \times C_p$.
\end{thm}

\subsection{S-rings}\label{subsec:schur}

Let $\mathcal{A}$ be a S-ring over $H$.
We say that a subgroup $K$ of $H$ is an \textit{$\mathcal{A}$-subgroup} if $\underline{K} \in \mathcal{A}$.
For each $\mathcal{A}$-subgroup $E$ of $H$, one can define a subring $\mathcal{A}_E$ by setting
\[\mathcal{A}_E = \mathcal{A} \cap \mathbb{Q}E.\]
It is easy to see that $\mathcal{A}_E$ is a S-ring over $E$
and \[\mathrm{Bsets}(\mathcal{A}_E) = \{ T \mid T \in \mathrm{Bsets}(\mathcal{A}), T \subseteq E \}.\]

An automorphism of $\mathcal{C}(\mathcal{A})$ taking $1_H$ to $1_H$ is called the \textit{automorphism} of $\mathcal{A}$,
which is denoted by $\mathrm{Aut}(\mathcal{A})$.
One can see that
\[\mathrm{Aut}(\mathcal{A}) = \mathrm{Aut}(H, R_H(\mathrm{Bsets}(\mathcal{A})))_{1_H},\]
\[\mathrm{Aut}(H, R_H(\mathrm{Bsets}(\mathcal{A}))) = \mathrm{Aut}(\mathcal{A}) H_R,\]
where $H_R$ is the right regular representation of $H$.

\begin{lem}\label{lem:con-number}
Let $\mathcal{A}$ be a S-ing over $H$. Let $T_1, T_2$ be elements of $\mathrm{Bsets}(\mathcal{A})$.
Then the following are equivalent.
\begin{enumerate}
\item $\underline{T_1} \cdot \underline{T_2} = \sum_{T \in \mathrm{Bsets}(\mathcal{A})} a_{R_H(T_1) R_H(T_2) R_H(T)} \underline{T}$.
\item $R_H(T_1) R_H(T_2) = \sum_{T \in \mathrm{Bsets}(\mathcal{A})} a_{R_H(T_1) R_H(T_2) R_H(T)} R_H(T)$.
\end{enumerate}
\end{lem}

Based on Theorem \ref{mk}, we state the following results in \cite{zies}.
\begin{prop}\label{prop:3}
Let $\mathcal{A}$ be a S-ring over $H$ and $m$ an element of $H$.
If $T, \{m \} \in \mathrm{Bsets}(\mathcal{A})$, then $Tm = \{ tm \mid t \in T \}$ lies in $\mathrm{Bsets}(\mathcal{A})$.
\end{prop}

\begin{prop}\label{prop:st}
Let $\mathcal{A}$ be a S-ring over $H$. If $T \in \mathrm{Bsets}(\mathcal{A})$,
then the stabilizer $\mathrm{St}_R(T):=\{ h \in H \mid Th=T \}$ is an $\mathcal{A}$-subgroup of $H$.
\end{prop}

\begin{prop}\label{prop:pre1}
Let $\mathcal{A}$ be a $p$-S-ring over a group $H$ of order $p^m$. Then
\begin{enumerate}
\item the group $\Ot(\mathcal{A}):=\{ h \in H \mid \{ h \} \in \mathrm{Bsets}(\mathcal{A}) \}$ is a nontrivial $\mathcal{A}$-subgroup,
\item the group $\Or(\mathcal{A}):=\langle \{ T^{-1}T \mid T \in \mathrm{Bsets}(\mathcal{A}) \} \rangle$ is a proper $\mathcal{A}$-subgroup,
\item there exists a series $H_0 = \{1_H \} < H_1 < \cdots < H_m = H$ of $\mathcal{A}$-subgroups such that $[H_{i+1}:H_i]=p$.
\end{enumerate}
\end{prop}

\begin{rem}
$\Ot(\mathcal{A})$ and $\Or(\mathcal{A})$ correpond to the thin radical and thin residue in the theory of association schemes, respectively.
\end{rem}

\section{Schurity of $p$-S-rings}\label{sec:main}
It is known that every $2$-S-ring over a group of order $8$ is commutative and Schurian (see \cite{hanakimi}).
From now on, we assume that $p$ is an odd prime.
It is well known that there are two non-abelian groups of order $p^3$ up to isomorphism, namely
\[ H_1 = \langle a, b \mid a^{p^2} = b^p = 1, ab = ba^{p+1} \rangle,\]
\[ H_2 = \langle a, b, c \mid a^p=b^p=c^p=1, [a,b]=c, [a,c]=[b,c]=1 \rangle . \]



Throughout this section, we assume that $\mathcal{A}$ is a non-commutative $p$-S-ring over $H_i$ $(i=1, 2)$ such that
\begin{equation}\label{A}
|T|= p  ~~\text{for each}~~ T \in \mathrm{Bsets}(\mathcal{A}) \setminus \mathrm{Bsets}(\mathcal{A}_{\Or(\mathcal{A})}),
\end{equation}
\begin{equation}\label{B}
| \{ \mathrm{St}_R(T) \mid T \in \mathrm{Bsets}(\mathcal{A}) \setminus \mathrm{Bsets}(\mathcal{A}_{\Or(\mathcal{A})}) \} | = p-1.
\end{equation}

By Theorem \ref{kim11}, we have
\[\Or(\mathcal{A}) = \Ot(\mathcal{A}) \cong C_p \times C_p.\]

Without loss of generality, we can assume $\Or(\mathcal{A}) = \langle a^p, b \rangle$ over $H_1$ and $\Or(\mathcal{A}) = \langle b, c \rangle$ over $H_2$.
For convenience, we omit the subindex $i$ of $H_i$ and denote $\Or(\mathcal{A})$ by $L$.

\begin{lem}\label{lem:center}
For each $T \in \mathrm{Bsets}(\mathcal{A}) \setminus \mathrm{Bsets}(\mathcal{A}_L)$,
$\mathrm{St}_R(T)$ is not the center $Z(H)$ of $H$.
\end{lem}
\begin{proof}
Suppose that $\mathrm{St}_R(T) = Z(H)$ for some $T \in \mathrm{Bsets}(\mathcal{A}) \setminus \mathrm{Bsets}(\mathcal{A}_L)$.
Then $T = hZ(H)$ for some $h \in H \setminus L$.
Note that $h = a^i b^j$ for some $1 \leq i \leq p-1$ and $0 \leq j \leq p-1$.
Since $\underline{T} \cdot \underline{T} = p \cdot \underline{h^2Z(H)}$ and $|h^2Z(H)|=p$, we have $h^2Z(H) \in \mathrm{Bsets}(\mathcal{A}) \setminus \mathrm{Bsets}(\mathcal{A}_L)$ by (\ref{A}).
Since $b^j \in \mathrm{Bsets}(\mathcal{A})$ for each $1 \leq j \leq p-1$,
it follows from Proposition \ref{prop:3} that $b^jT \in \mathrm{Bsets}(\mathcal{A}) \setminus \mathrm{Bsets}(\mathcal{A}_L)$.
These imply that every element of $\mathrm{Bsets}(\mathcal{A}) \setminus \mathrm{Bsets}(\mathcal{A}_L)$ has the form
$a^ib^jZ(H)$ for some $1 \leq i \leq p-1$, $0 \leq j \leq p-1$.
By Proposition \ref{prop:pre1}(ii), we have $L = Z(H)$, which contradicts $L \cong C_p \times C_p$.
\end{proof}

\begin{lem}[See \cite{hk}]\label{orderp4}
If $\mathcal{A}$ is Schurian, then the order of $\mathrm{Aut}(\mathcal{A})$ is $p$.
\end{lem}
\begin{proof}
Let $(X,S)=(H, R_H(\mathrm{Bsets}(\mathcal{A}))$.
For convenience, we denote $\mathrm{Aut}(X, S)$ by $G$.
By the orbit-stabilizer property, we have
$|G| =  |G_\alpha| |X|$ for a fixed $\alpha \in X$.
Since $|X|=p^3$, it suffices to verify $|G_\alpha|=p$.

Let $r_1$ is an element of $S \setminus \Ot(S)$.
We fix an element $\beta$ of $X$ such that $\beta \in \alpha r_1$.
Under the assumption that $(X, S)$ is Schurian,
we have
\[|G_\alpha| = |G_{\alpha, \beta}||\alpha r_1| = |G_{\alpha, \beta}|p\]
by the orbit-stabilizer property.
We shall show $|G_{\alpha, \beta}|=1$.

Let $\gamma$ be an element of $\alpha r_1$ such that $\gamma \in \beta t$ for some $t \in \{ s \in S \mid r_1 s = r_1 \} \setminus \{1_X\}$.
$G_{\alpha, \beta}$ fixes $\gamma$ since $a_{r_1 t r_1} = 1$. Thus, $G_{\alpha, \beta}$ fixes each element of $\alpha r_1$.

We consider an arbitrary element $r_2 \in S \setminus \Ot(S)$ such that $r_1 \neq r_2$.
For $\delta \in \alpha r_2$, we assume that there exists $g \in G_{\alpha, \beta}$ such that $\delta^g \neq \delta$.
Then there exist $s_1, s_2 \in S$ such that $(\beta, \delta), (\beta, \delta^g) \in s_1$ and $(\gamma, \delta), (\gamma, \delta^g) \in s_2$.
This means $a_{s_1 s_2^\ast t} \geq 2$ and $a_{r_2 s_1^\ast r_1} \geq 2$.

We divide our consideration into the following cases.
\begin{enumerate}
\item If $r_1 \Or(S) = r_2 \Or(S)$, then $s_1, s_2 \in \Ot(S)$.
This contradicts $a_{s_1 s_2^\ast t} = 1$.
Thus, $G_{\alpha, \beta}$ fixes each element of $\alpha r_2$. This implies that $G_{\alpha, \beta}$ fixes each element of $\alpha s$ with $r_1 \Or(S) = s \Or(S)$.

\item If $r_1 \Or(S) \neq r_2 \Or(S)$, then $s_1, s_2 \in S \setminus \Ot(S)$.
This contradicts $a_{r_2 s_1^\ast r_1} =1$.
\end{enumerate}
Whichever the case may be, $G_{\alpha, \beta}$ fixes each element of $\alpha s$ for all $s \in S$.
Hence, we have $|G_{\alpha, \beta}|=1$.
\end{proof}

In the rest of this section, we consider the association scheme $(H, R_H(\mathrm{Bsets}(\mathcal{A})))$ corresponding to the S-ring $\mathcal{A}$.
We fix an element $z$ of $Z(H)$.

\begin{lem}\label{per:induce}
Let $T_1, \dotsc, T_{p-1}$ be elements of $\mathrm{Bsets}(\mathcal{A}) \setminus \mathrm{Bsets}(\mathcal{A}_L)$ such that
$\mathrm{St}_R(T_i) \neq \mathrm{St}_R(T_j)$ for distinct $1 \leq i,j \leq p-1$.
Then we have
\begin{equation}\label{bsets-ordering}
\mathrm{Bsets}(\mathcal{A}) = \{ \{l \} \mid l \in L \} \cup \bigcup_{1 \leq i \leq p-1} \{ T_iz^j \mid 0 \leq j \leq p-1 \}.
\end{equation}
\end{lem}
\begin{proof}
By Lemma \ref{lem:center}, we have $T_i \neq T_i z^j$ for each $z^j \in Z(H)$.
It follows from Proposition \ref{prop:3} that $T_i z \in \mathrm{Bsets}(\mathcal{A})$.
This completes the proof.
\end{proof}

From now on, we assume that for each $R_H(T) \in R_H(\mathrm{Bsets}(\mathcal{A}))$,
the rows and columns of the adjacency matrix of $R_H(T)$ are indexed by the following:
\begin{enumerate}
\item Assume that each $T_i$ has a fixed ordering of elements in itself;
\item Assume that $\mathrm{Bsets}(\mathcal{A}_L)$ has a fixed ordering of elements in itself;
\item A fixed ordering of $L$ and $T_i$'s is
\[(L, T_1, T_1z, \dotsc T_1z^{p-1}, T_2, T_2z, \dotsc T_2z^{p-1}, \dotsc, T_{p-1}, T_{p-1}z, \dotsc, T_{p-1}z^{p-1}),\]
\end{enumerate}
where $L$ and $T_iz^j$ ($1 \leq i \leq p-1, 0 \leq j \leq p-1$) are given in (\ref{bsets-ordering}).

\begin{lem}\label{per}
Let $T, T'$ and $T''$ be elements of $\mathrm{Bsets}(\mathcal{A}) \setminus \mathrm{Bsets}(\mathcal{A}_L)$ such that
\[R_H(T) \cap (T' \times T'') \neq \emptyset.\]
Then we have the following:
\begin{enumerate}
\item The adjacency matrix of $R_H(T) \cap (T' \times T'')$ is a permutation matrix;
\item $R_H(Tz^k) \cap (T' \times T'') \neq \emptyset$ for each $1 \leq k \leq p-1$.
\end{enumerate}
\end{lem}
\begin{proof}
It follows from Lemma \ref{lem:inter-relation} that $R_H(T) \in R_H(T')^\ast R_H(T'')$.
Put $\Or(R_H):= \Or(R_H(\mathrm{Bsets}(\mathcal{A})))$.
Since $R_H(T) \cap (T' \times T'') \neq \emptyset$, we have $R_H(T')\Or(R_H) \neq R_H(T'')\Or(R_H)$.
This means $\mathrm{St}_R(T') \neq \mathrm{St}_R(T'')$.
This implies $A_{R_H(T')^\ast} A_{R_H(T'')} = \sum_{h \in R_H(T)\Or(R_H)} A_h$.
It follows from Lemmas \ref{lem:intersection} that
the adjacency matrix of $R_H(T) \cap (T' \times T'')$ is a permutation matrix.
It follows from Lemma \ref{lem:inter-relation} that $R_H(Tz^k) \cap (T' \times T'') \neq \emptyset$ for each $1 \leq k \leq p-1$.
\end{proof}

From now on, for $E, F, G \subseteq H$, we denote $\{(eg, fg) \mid e \in E, f \in F, g \in G \}$ by $(E \times F)^G$.

\vskip5pt
We introduce a notion of compatibility given in \cite{hkp,kim20130401}.
Denote by $r(\beta, \gamma)$ the element of $R_H(\mathrm{Bsets}(\mathcal{A}))$ containing $(\beta, \gamma)$.
We define the subset $\Gamma_1$ of the symmetric group on $H$ with respect to $\mathrm{Bsets}(\mathcal{A})$ such that
\begin{enumerate}
\item for $\sigma \in \Gamma_1$, $L$ is the set of fixed points of $\sigma $,
\item for $\sigma \in \Gamma_1$ and $T \in \mathrm{Bsets}(\mathcal{A})$ with size $p$, $\sigma$ is a $p$-cycle on $T$ such that $r(\beta, \gamma) =  r(\beta^\sigma, \gamma^\sigma)$ for each $(\beta, \gamma) \in T \times T$.
\end{enumerate}

For $T, T' \in \mathrm{Bsets}(\mathcal{A}) $ and $\sigma \in \Gamma_1$, we say that $T$ and $T'$ are
\textit{compatible with respect to $\sigma$} if
\[ r(\beta, \gamma) = r(\beta^\sigma, \gamma^\sigma) ~\text{for each}~ (\beta, \gamma) \in T \times T'.\]

We shall write $T \sim_\sigma T'$ if $T$ and $T'$ are compatible with respect to $\sigma$, otherwise $T \nsim_\sigma T'$.

\begin{thm}\label{mainthm1}
Let $T_1, T_2, \dotsc, T_{p-1}$ be elements of $\mathrm{Bsets}(\mathcal{A}) \setminus \mathrm{Bsets}(\mathcal{A}_L)$
such that
\[\mathrm{St}_R(T_i) \neq \mathrm{St}_R(T_j)\]
for distinct $1 \leq i, j \leq p-1$.
Then $\mathcal{A}$ is Schurian if and only if
there exists a nontrivial element $\sigma \in \Gamma_1$ such that
$T_i \sim_\sigma T_j$ for all $1 \leq i, j \leq p-1$.
\end{thm}
\begin{proof}
We consider $(H, R_H(\mathrm{Bsets}(\mathcal{A})))$ corresponding to $\mathcal{A}$.
By Lemma \ref{orderp4}, $\mathcal{A}$ is Schurian if and only if there exists a nontrivial element $\sigma \in \Gamma_1$
such that $T \sim_\sigma T'$ for all $T, T' \in \mathrm{Bsets}(\mathcal{A})$.
It suffices to prove the sufficiency.
Note that $T \sim_\sigma T'$ for all $\sigma \in \Gamma_1$, $T \in \mathrm{Bsets}(\mathcal{A}_L)$ and $T' \in \mathrm{Bsets}(\mathcal{A}) \setminus \mathrm{Bsets}(\mathcal{A}_L)$.

\vskip5pt
From $\sigma$ given in the assumption,
we find a nontrivial element $\sigma'$ of $\Gamma_1$ such that $T \sim_{\sigma'} T'$ for all $T, T' \in \mathrm{Bsets}(\mathcal{A})$.

Define $\sigma' \in Sym(H)$ such that
\begin{enumerate}
\item $\sigma'|_{T_i} = \sigma|_{T_i}$ for all $1 \leq i \leq p-1$,
\item $h^{\sigma'}z^l = (hz^l)^{\sigma'}$ for all $h \in T_i (1 \leq i \leq p-1), 1 \leq l \leq p-1$.
\end{enumerate}

First of all, we verify $\sigma' \in \Gamma_1$.
For a given $T_iz^l$, we show $r(\beta, \gamma) = r(\beta^{\sigma'}, \gamma^{\sigma'})$ for each $(\beta, \gamma) \in T_iz^l \times T_iz^l$.
There exists $T \in \mathrm{Bsets}(\mathcal{A}_L)$ such that $(\beta, \gamma) \in R_H(T) \cap (T_iz^l \times T_iz^l)$.
Since $(R_H(T) \cap (T_i \times T_i))^{z^l} = R_H(T) \cap (T_iz^l \times T_iz^l)$,
there exists $(\beta_0, \gamma_0) \in R_H(T) \cap (T_i \times T_i)$ such that $(\beta_0, \gamma_0)^{z^l} = (\beta, \gamma)$.
Thus, it follows from the definition of $\sigma'$ that
$(\beta^{\sigma'}, \gamma^{\sigma'}) = ((\beta_0z^l)^{\sigma'}, (\gamma_0z^l)^{\sigma'}) = (\beta_0^{\sigma'}z^l, \gamma_0^{\sigma'}z^l) = (\beta_0^{\sigma}z^l, \gamma_0^{\sigma}z^l) \in R_H(T)^{z^l} =  R_H(T)$.

\vskip5pt
Next, we verify $T_i \sim_{\sigma'} T_i z^l$ for all $1 \leq l \leq p-1$.
Let $(\beta, \gamma) \in T_i \times T_iz^l$.
Then there exists $T \in \mathrm{Bsets}(\mathcal{A}_L)$ such that $(\beta, \gamma) \in R_H(T) \cap (T_i \times T_iz^l)$,
and we have $\gamma = \beta tz^m$ for some $t \in \mathrm{St}_R(T)$ and $m$.
Since $\beta t = q \beta$ for some $q \in H$,
we have $(\beta, \gamma) =  (\beta, \beta tz^m) = (\beta, qz^m \beta ) \in r(1, qz^m)$.
It follows from the definition of $\sigma'$ that
$(\beta^{\sigma'}, \gamma^{\sigma'}) = (\beta^\sigma, (\beta t)^\sigma z^m) \in
r(\beta^\sigma, (\beta t)^\sigma) r((\beta t)^\sigma, (\beta t)^\sigma z^m) = r(\beta, \beta t) r((\beta t)^\sigma, (\beta t)^\sigma z^m)
= r(\beta, q \beta) r(1, z^m) = r(1, q) r(q, qz^m) = r(1, qz^m)$.
Thus, $T_i \sim_{\sigma'} T_i z^l$ for all $1 \leq l \leq p-1$.

By the same argument, for each $T_i z^l$ we have $T_i  z^l \sim_{\sigma'} T_i z^l z^m$ for all $1 \leq m \leq p-l-1$.
This implies that $T_i z^l\sim_{\sigma'} T_i z^n$ for all $0 \leq l, n \leq p-1$.

\vskip5pt
Finally, we verify $T_i z^l \sim_{\sigma'} T_j z^n$ for distinct $1 \leq i, j \leq p-1$ and all $1 \leq l, n \leq p-1$.

Now we say that $(T_i \times T_j)$ is equivalent to $(T_i z^l \times T_j z^n)$
if the adjacency matrix of $R_H(T) \cap (T_i \times T_j)$ concides with that of $R_H(T') \cap (T_i z^l \times T_j z^n)$
for some $T, T' \in \mathrm{Bsets}(\mathcal{A}) \setminus \mathrm{Bsets}(\mathcal{A}_L)$.
For convenience, we denote it by $(T_i \times T_j) \simeq (T_i z^l \times T_j z^n)$.
We show that $(T_i \times T_j)$ is equivalent to $(T_i z^l \times T_j z^n)$ for all $0 \leq l, n \leq p-1$.

\vskip5pt
\textbf{Claim 1}: $(T_i \times T_j)  \simeq  (T_i z^l \times T_j z^l)$.

There exists $T \in \mathrm{Bsets}(\mathcal{A}) \setminus \mathrm{Bsets}(\mathcal{A}_L)$ such that
$R_H(T) \cap (T_i \times T_j) \neq \emptyset$.

Let $(\beta,\gamma)$ be an element of $R_H(T) \cap (T_i \times T_j)$.
Then $(\beta z^l,\gamma z^l) \in T_iz^l \times T_jz^l$.
By the definition of $R_H(T)$, we have $(\beta z^l,\gamma z^l) \in R_H(T)$.
Since the orderings of $T_iz^l$ and $T_jz^l$ are depending on $T_i$ and $T_j$ respectively,
it is easy to see that the adjacency matrix of $R_H(T) \cap (T_i \times T_j)$ coincides with that of $R_H(T) \cap (T_iz^l \times T_jz^l)$.
This completes the proof of Claim $1$.

\vskip5pt
\textbf{Claim 2}: $(T_i \times T_j)  \simeq  (T_i \times T_j z^m)$.

There exists $T \in \mathrm{Bsets}(\mathcal{A}) \setminus \mathrm{Bsets}(\mathcal{A}_L)$ such that
$R_H(T) \cap (T_i \times T_j) \neq \emptyset$.

For $T, Tz^m \in \mathrm{Bsets}(\mathcal{A}) \setminus \mathrm{Bsets}(\mathcal{A}_L)$, we have
\[R_H(T) \cap (T_i \times T_j) = (\{1_H\} \times T)^{T_i} \cap (T_i \times T_j) ~~\text{and}\]
\[R_H(Tz^m) \cap (T_i \times T_jz^m) = (\{1_H\} \times Tz^m)^{T_i} \cap (T_i \times T_jz^m).\]

This implies that the adjacency matrix of $R_H(T) \cap (T_i \times T_j)$ coincides with that of $R_H(Tz^m) \cap (T_i \times T_jz^m)$.
This completes the proof of Claim $2$.

\vskip5pt
By Claims $1$ and $2$, we have $(T_i \times T_j)  \simeq  (T_i z^l \times T_j z^l) \simeq (T_i z^l \times T_j z^{l + (m-l)})$.
This means that $(T_i \times T_j)  \simeq  (T_i z^l \times T_j z^n)$ for all $0 \leq l, n \leq p-1$.
So, it is easy to see that $T_i z^l \sim_{\sigma'} T_j z^n$.

\vskip5pt
Therefore, we have $T \sim_{\sigma'} T'$ for all $T, T' \in \mathrm{Bsets}(\mathcal{A}) \setminus \mathrm{Bsets}(\mathcal{A}_L)$.
\end{proof}

\section{Construction of suitable sequences}\label{sec:comb}

In this section, we introduce a notion related to construction of $p$-S-rings.

Let $\mathbb{Z}_p = \{ 0, 1, \dotsc, p-1 \}$.
A sequence $(x_1, x_2, \dotsc, x_{p-1})$ is called \textit{suitable} if $\{ x_1, x_2, \dotsc, x_{p-1} \} = \mathbb{Z}_p \setminus \{l\}$
for some $l \in \mathbb{Z}_p$ such that
$x_1 =0$, $x_i + i \equiv x_{p-i}$ $(\mathrm{mod}~ p)$ for each $1 \leq i \leq \frac{p-1}{2}$.

\begin{ex}\label{ex:seq}
In $\mathbb{Z}_p$, $(x_i \mid 1 \leq i \leq p-1)$ is suitable, where $x_i \equiv (\frac{p-1}{2})(i-1) ~~(\mathrm{mod}~ p)$.
\end{ex}

\begin{ex}\label{ex:p7}
In $\mathbb{Z}_7$, $(0, 4, 2, 5, 6, 1)$ and $(0, 2, 3, 6, 4, 1)$ are suitable.
\end{ex}

\begin{ex}\label{ex:p11}
In $\mathbb{Z}_{11}$, $(0, 6, 10, 3, 4, 9, 7, 2, 8, 1)$ and $(0, 4, 10, 5, 3, 8, 9, 2, 6, 1)$ are suitable.
\end{ex}

\begin{prob}
Find all of suitable sequences of $\mathbb{Z}_p$, where $p$ is an odd prime.
\end{prob}

\vskip10pt
In the case of $p \equiv 3 ~~(\mathrm{mod}~ 4)$, we construct a suitable sequence.
The following proposition is a generalization of first cases in Examples \ref{ex:p7} and \ref{ex:p11}, which is different from Example \ref{ex:seq}.

\begin{prop}\label{prop:newseq}
For each odd prime $p=4k+3$, there exists a suitable sequence with $x_2 = \frac{p+1}{2}$.
\end{prop}
\begin{proof}
Let $( x_i \mid 1 \leq i \leq p-1)$ be the suitable sequence given in Example \ref{ex:seq}.
First of all, we observe the following fact.

\vskip5pt
\textbf{Claim 1}: $(  x_2, x_4, \dotsc, x_{2k+2}) = (2k+1, 2k \dotsc, k+1)$ and
\[(x_{p-2k-2}, \dotsc, x_{p-4}, x_{p-2} ) = (3k+3, \dotsc, 2k+4, 2k+3).\]
By direct calculation, we have $x_2 = \frac{p-1}{2} = 2k+1$ and $x_{p-2} = 2k + 3$.
Since $x_i - x_{i+2} \equiv  (\frac{p-1}{2})(i-1) - (\frac{p-1}{2})(i+1)  \equiv 1 ~~(\mathrm{mod}~ p)$,
we have $x_4 = 2k$ and $x_{p-4} = 2k+ 4$.
This implies Claim $1$.

\vskip5pt
We denote $\{ x_2, x_4, \dotsc, x_{2k+2}, x_{p-2k-2}, \dotsc, x_{p-4}, x_{p-2} \}$ by $\mathcal{C}$.
Note that \[\mathcal{C} \cup \{ \frac{p+1}{2} \} = \{x \mid k+1 \leq x \leq 3k+3 \} ~~\text{and}\]
\[\{ x_{p-i} - x_i \mid x_i, x_{p-i} \in \mathcal{C}, 2 \leq i \leq 2k \} \cup \{ x_{p-2k-2} - x_{2k+2} \}= \{ 2, 4, \dotsc, 2k, 2k+1\}.\]

\vskip5pt
Rearranging $2k +2$ elements in $\mathcal{C} \cup \{ \frac{p+1}{2} \}$, we construct a new suitable sequence.

\vskip5pt
\textbf{Claim 2}: There exists a sequence $(y_2, y_4, \dotsc, y_{2k},$ $y_{p-2k-2}, y_{2k+2}, y_{p-2k}$ $\dotsc, y_{p-4}, y_{p-2})$
such that
\[y_{p - 2l} - y_{2l} \equiv 2l ~~(\mathrm{mod}~ p) ~~\text{for}~~ 1 \leq l \leq k, \]
\[y_{2k+2} - y_{p-2k-2} \equiv 2k+1 ~~(\mathrm{mod}~ p)  ~~\text{and} \]
\[ \frac{p-1}{2} \not\in \{ y_2, y_4, \dotsc, y_{2k}, y_{p-2k-2}, y_{2k+2}, y_{p-2k}  \dotsc, y_{p-4}, y_{p-2} \}.\]

First of all, we consider the case that $k$ is even.
Put \[(y_{2k+2}, y_{p-2k-2}) = (x_{p-2k-4}, x_{2k}) ~~\text{ and }~~ (y_{p-2}, y_2) = (x_{p-4}, \frac{p+1}{2}).\]
Then we have $y_{2k+2} - y_{p-2k-2} \equiv 2k + 1 ~~(\mathrm{mod}~ p) ~~\text{ and }~~ y_{p-2} - y_2 \equiv 2 ~~(\mathrm{mod}~ p)$.

For each $2 \leq l \leq k$, put \[(y_{p - 2l}, y_{2l})= (x_{p - 2(l-1)}, x_{2(l+1)})  ~~(l : \text{even}),\]
\[(y_{p - 2l}, y_{2l})= (x_{p - 2(l+1)}, x_{2(l-1)}) ~~(l : \text{odd}).\]

Computing $y_{p - 2l} - y_{2l}$ based on Claim $1$, we have $y_{p - 2l} - y_{2l} \equiv 2l$ $(\mathrm{mod}~ p)$.

\vskip5pt
Next, we consider the case that $k$ is odd.
Put \[(y_{2k+2}, y_{p-2k-2}) = (x_{p-2k}, x_{2k+2}) ~~\text{ and }~~ (y_{p-2}, y_2) = (x_{p-4}, \frac{p+1}{2}).\]
For each $2 \leq l \leq k$, put \[(y_{p - 2l}, y_{2l})= (x_{p - 2(l+1)}, x_{2(l-1)}) ~~(l : \text{even}),\]
\[(y_{p - 2l}, y_{2l})= (x_{p - 2(l-1)}, x_{2(l+1)}) ~~(l : \text{odd}).\]
Then we have
$y_{2k+2} - y_{p-2k-2} \equiv 2k+1 ~~(\mathrm{mod}~ p)$, $y_{p-2} - y_2 \equiv 2 ~~(\mathrm{mod}~ p)$ and
$y_{p - 2l} - y_{2l} \equiv 2l ~~(\mathrm{mod}~ p) ~~\text{for}~~ 1 \leq l \leq k$.
This completes the proof of Claim $2$.

\vskip5pt
Therefore, it follows from Claims $1$ and $2$ that $(x_1, y_2, \dotsc, x_{2k-1},$ $y_{2k}, y_{2k+1},$ $y_{2k+2}, y_{2k+3},$ $x_{2k+4}, \dotsc, y_{p-2}, x_{p-1})$ is suitable.
\end{proof}

\begin{prop}\label{prop:suit}
Let $\mathcal{A}$ be the vector space spanned by $\{ \underline{T} \mid T \in \{ \{l\} \mid l \in L \} \cup \{T_i a^{pj} \mid 1 \leq i \leq p-1, 0 \leq j \leq p-1 \} \}$ such that $L = \langle a^p, b \rangle$, $T_i = a^i \langle ba^{px_i} \rangle$ $( 1 \leq i \leq p-1 )$ and $(x_1, x_2, \dotsc, x_{p-1})$ is suitable.
Then $\mathcal{A}$ is a $p$-S-ring over $H_1$.
\end{prop}
\begin{proof}
Put $\mathrm{Bsets(\mathcal{A})} = \{ \{l\} \mid l \in L \} \cup \{T_i a^{pj} \mid 1 \leq i \leq p-1, 0 \leq j \leq p-1 \}$.
Since $x_i + i \equiv x_{p-i}$ $(\mathrm{mod}~ p)$, we have
\begin{eqnarray*}
T_i^{-1}      & = &  \langle ba^{px_i} \rangle a^{-i}  =  \langle ba^{px_i} \rangle a^{p-i} a^{p^2 -p} \\
              & = &   a^{p-i} \langle ba^{(p^2-p)(p-i)}a^{px_i} \rangle a^{p(p-1)}          \\
              & = &   a^{p-i} \langle ba^{pi}a^{px_i} \rangle a^{p(p-1)}  =   a^{p-i} \langle ba^{p(i+ x_i)} \rangle a^{p(p-1)}    \\
              & = &   a^{p-i} \langle ba^{px_{p-i}} \rangle a^{p(p-1)}  =  T_{p-i}a^{p(p-1)} \in \mathrm{Bsets(\mathcal{A})}.
\end{eqnarray*}

For $T_i, T_j \in \mathrm{Bsets(\mathcal{A})}$,
\begin{eqnarray*}
T_i T_j      & = &   a^i \langle ba^{px_i} \rangle  a^j \langle ba^{px_j} \rangle = \cdots = a^{i+j} \langle ba^{p(x_i-j)} \rangle  \langle ba^{px_j} \rangle \\
             & = & \left\{
                      \begin{array}{ll}
                      a^p \langle ba^{px_j} \rangle & \hbox{if $j = p-i$;} \\
                      a^{i+j}L & \hbox{otherwise.}
                      \end{array}
                     \right.
\end{eqnarray*}
This means that $\underline{T_i} \cdot \underline{T_j}$ is written as a linear combination of the elements in $\{ \underline{T} \mid T \in \mathrm{Bsets(\mathcal{A})}\}$.

We leave to the reader to check the remaining part of proof.
\end{proof}

Similarly, we have the following.

\begin{prop}\label{prop:suit2}
Let $\mathcal{A}$ be the vector space spanned by $\{ \underline{T} \mid T \in \{ \{l\} \mid l \in L \} \cup \{T_i c^{j} \mid 1 \leq i \leq p-1, 0 \leq j \leq p-1 \} \}$ such that $L = \langle b, c \rangle$, $T_i = a^i \langle bc^{x_i} \rangle$ $( 1 \leq i \leq p-1 )$ and $(x_1, x_2, \dotsc, x_{p-1})$ is suitable.
Then $\mathcal{A}$ is a $p$-S-ring over $H_2$.
\end{prop}

\section{Non-schurian $p$-S-rings arisen from suitable sequences}\label{sec:nonschur}
In this section, we show that a $p$-S-ring over $H_1$ arisen from a suitable sequence is non-Schurian.
For convenience, we omit the subindex $1$ of $H_1$
By Proposition \ref{prop:newseq}, there exists a suitable sequence such that
\[x_1 = 0, x_2 =\frac{p+1}{2}, x_3=p-1, \dotsc, x_{p-1}=1.\]

Using Proposition \ref{prop:suit} and this sequence, one can get a $p$-S-ring $\mathcal{A}$ such that
\begin{equation}\label{bsetsa}
\mathrm{Bsets}(\mathcal{A})=  \{ \{l\} \mid l \in L \} \cup \{T_i a^{pj} \mid 1 \leq i \leq p-1, 0 \leq j \leq p-1 \}.
\end{equation}

By Theorem \ref{mainthm1}, if $\mathcal{A}$ is Schurian, then there exists an element $\sigma \in \Gamma_1$ such that
$T_1 \sim_\sigma T_2 \sim_\sigma T_3 \sim_\sigma T_1$.
In the rest of this section, we prove that any $\sigma \in \Gamma_1$ with $T_1 \sim_\sigma T_2 \sim_\sigma T_3$ does not satisfy $T_1 \sim_\sigma T_3$.

From now on, we assume that for all $1 \leq i \leq p-1$,
\[(a^i, a^iba^{px_i}, a^i(ba^{px_i})^2, \dotsc, a^i(ba^{px_i})^{p-1})\]
is a fixed ordering of $T_i$.

\begin{lem}\label{sec5:123com}
Let $T_i= a^i \langle ba^{px_i} \rangle$, $T_j= a^j \langle ba^{px_j} \rangle$ $(i \leq j)$ and $T_k= a^k \langle ba^{px_k} \rangle$
be elements of $(\ref{bsetsa})$ such that
$R_H(T_k) \cap (T_i \times T_j) \neq \emptyset$.
Then we have
\begin{equation}\label{imequal}
(x_i - x_k - (p-1)i)n \equiv (x_j - x_k - (p-1)i)l ~~(\mathrm{mod}~ p).
\end{equation}
\end{lem}
\begin{proof}
By Lemma \ref{per}, the adjacency matrix of $R_H(T_k) \cap (T_i \times T_j)$ is a permutation matrix.
For each element $\beta$ of $T_j$, there exists a unique $\alpha \in T_i$ such that $(\alpha, \beta) \in R_H(T_k) \cap (T_i \times T_j)$.
By direct calculation, we have
\begin{eqnarray*}
R_H(T_k) \cap (T_i \times T_j) &=&  (\{1_H\} \times a^k \langle ba^{px_k} \rangle)^{a^i \langle ba^{px_i} \rangle} \cap (a^i \langle ba^{px_i} \rangle \times a^j \langle ba^{px_j} \rangle).
\end{eqnarray*}
Note that $k+i=j$.
 For each $0 \leq l \leq p-1$, there exist some $m, n$ such that
\[( ba^{px_k} a^{p(p-1)i} )^m  ( ba^{px_i} )^n = (  ba^{px_j} )^l.\]

The above equation implies
\begin{enumerate}
\item $(x_k+(p-1)i)m + x_in \equiv x_jl$  $(\mathrm{mod}~ p)$,
\item $m + n \equiv l$ $(\mathrm{mod}~ p)$.
\end{enumerate}
It follows from the system of congruences that
\[(x_i - x_k - (p-1)i)n \equiv (x_j - x_k - (p-1)i)l ~~(\mathrm{mod}~ p).\]
\end{proof}

\begin{rem}\label{permutation:cell}
By Lemma \ref{per}, the adjacency matrix of $R_H(T_k) \cap (T_i \times T_j) \neq \emptyset$ is a permutation matrix.
In (\ref{imequal}), there exists a unique $n$ for a given $l$.
This means that for each element of $T_j$,
it is possible to determine a unique element of $T_i$ with respect to $R_H(T_k)$.
\end{rem}

\vskip10pt
Using Lemma \ref{sec5:123com}, we determine the adjacency matrices
of $R_H(T_1) \cap (T_1 \times T_2)$, $R_H(T_1) \cap (T_2 \times T_3)$ and $R_H(T_2) \cap (T_1 \times T_3)$,
where $T_1, T_2, T_3 \in \mathrm{Bsets}(\mathcal{A}) \setminus \mathrm{Bsets}(\mathcal{A}_L)$.

\vskip5pt
\textbf{Case 1}: $R_H(T_1) \cap (T_1 \times T_2)$.

In (\ref{imequal}), if $i=k=1$ and $j=2$, then we have
\[-(p-1)n \equiv (\frac{p+1}{2}-(p-1))l ~~(\mathrm{mod}~ p).\]
So, $n \equiv \frac{-p+3}{2}l ~~(\mathrm{mod}~ p)$.
The adjacency matrix of $R_H(T_1) \cap (T_1 \times T_2)$ is associated with
\begin{equation}\label{case1}
2n \equiv 3l ~~(\mathrm{mod}~ p).
\end{equation}

\vskip5pt
\textbf{Case 2}: $R_H(T_1) \cap (T_2 \times T_3)$.

In (\ref{imequal}), if $i=2, j=3$ and $k=1$, then we have
\[(\frac{p+1}{2} - (p-1)2)n \equiv ((p-1) - (p-1)2)l ~~(\mathrm{mod}~ p).\]
So, $\frac{-p+5}{2}n \equiv l$ $(\mathrm{mod}~ p)$.
The adjacency matrix of $R_H(T_1) \cap (T_2 \times T_3)$ is associated with
\begin{equation}\label{case2}
5n \equiv 2l ~~(\mathrm{mod}~ p).
\end{equation}

\vskip5pt
\textbf{Case 3}: $R_H(T_2) \cap (T_1 \times T_3)$.

In (\ref{imequal}), if $i=1, j=3$ and $k=2$, then we have
\[(-\frac{p+1}{2} - (p-1))n \equiv ((p-1) - \frac{p+1}{2} - (p-1))l ~~(\mathrm{mod}~ p).\]
So, $(p-1)n \equiv l ~~(\mathrm{mod}~ p)$.
The adjacency matrix of $R_H(T_2) \cap (T_1 \times T_3)$ is associated with
\begin{equation}\label{case3}
(p-1)n \equiv l ~~(\mathrm{mod}~ p).
\end{equation}

\vskip5pt
In (\ref{case1}), (\ref{case2}) and (\ref{case3}), if $l=0$, then $n=0$.
So, we have
\[(a,a^2) \in R_H(T_1), (a^2,a^3) \in R_H(T_1) ~~\text{and}~~ (a,a^3) \in R_H(T_2)\]
for $(a,a^2) \in T_1 \times T_2$, $(a^2,a^3) \in T_2 \times T_3$ and $(a,a^3) \in T_1 \times T_3$.

For each $l_i$ $(0 \leq i \leq p-1)$, there exists a unique $n_{l_i}$ with respect to (\ref{case1}).
So, it is possible to define a permutation $\sigma_1 \in \Gamma_1$ such that $T_1 \sim_{\sigma_1} T_2$ and
$\sigma_1|_{T_1} = (l_0, l_1, \dotsc, l_{p-1})$, $\sigma_1|_{T_2} = (n_{l_0}, n_{l_1}, \dotsc, n_{l_{p-1}})$.
For each $n_{l_i}$, applying the same argument for (\ref{case2}),
it is possible to determine a permutation $\sigma \in \Gamma_1$ such that $T_1 \sim_\sigma T_2 \sim_\sigma T_3$.
By Lemma \ref{orderp4} and Theorem \ref{mainthm1}, it suffices to check either $T_1 \sim_\sigma T_3$ or $T_1 \nsim_\sigma T_3$.

\vskip5pt
For the given $\sigma \in \Gamma_1$ with $T_1 \sim_\sigma T_2 \sim_\sigma T_3$, it induces the orbit of $(a,a^3)$ in $T_1 \times T_3$.
In order to describe it, we connect $2m \equiv 3l ~~(\mathrm{mod}~ p)$ and $5n \equiv 2m$ $(\mathrm{mod}~ p)$ given in (\ref{case1}) and (\ref{case2}).
So, we have
\begin{equation}\label{case123}
5n \equiv 3l ~~(\mathrm{mod}~ p).
\end{equation}
Note that $T_1 \sim_\sigma T_3$ if and only if (\ref{case3}) and (\ref{case123}) give the same set of orbits in $T_1 \times T_3$.

\vskip5pt
Suppose $T_1 \sim_\sigma T_3$.
Then $(n,l)=(p-1,1)$ satisfies (\ref{case3}).
Substituting $(p-1,1)$ for $(n,l)$ in (\ref{case123}), we obtain $5p \equiv 8 ~~(\mathrm{mod}~ p)$.
In order to hold $5p \equiv 8 ~~(\mathrm{mod}~ p)$, we shall show that $p$ must be even.
Suppose to the contrary that $p=2k+1$. Then
\begin{eqnarray*}
5(2k+1) \equiv 8 ~~(\mathrm{mod}~ 2k+1)  & \Leftrightarrow &  10k \equiv 3 ~~(\mathrm{mod}~ 2k+1) \\
                                         & \Leftrightarrow &  0 \equiv 8 ~~(\mathrm{mod}~ 2k+1).
\end{eqnarray*}

This is a contradiction.
But, $p=2k$ contradicts the assumption that $p$ is an odd prime.

\vskip10pt
Therefore, it is impossible to exist $\sigma \in \Gamma_1$ such that $T_1 \sim_\sigma T_2 \sim_\sigma T_3 \sim_\sigma T_1$.
We conclude that $\mathcal{A}$ is non-Schurian.

\begin{rem}
By a parallel argument, it is possible to give a non-Schurian $p$-S-ring over $H_2$.
\end{rem}

\vskip15pt
\textbf{Acknowledgement}

The author would like to thank anonymous referees for their valuable comments.

\bibstyle{plain}

\end{document}